\documentclass{amsart}
\usepackage{amssymb}
\usepackage{amsfonts}

\newtheorem{theorem}{Theorem}
\theoremstyle{plain}

\newtheorem{corollary}{Corollary}

\newtheorem{lemma}{Lemma}

\newtheorem{remark}{Remark}

\numberwithin{equation}{section}
\input{tcilatex}

\begin{document}
\title[Inequalities For $(\alpha ,m)-$ Convex functions]{Hermite-Hadamard
Type Inequalities via $(\alpha ,m)-$ Convexity}
\author{M. Emin \"{O}zdemir$^{\diamondsuit }$}
\address{$^{\diamondsuit }$ATAT\"{U}RK UNIVERSITY, K.K EDUCATION FACULTY,
DEPARTMENT OF MATHEMATICS, 25240, CAMPUS, ERZURUM, TURKEY}
\email{emos@atauni.edu.tr}
\author{Merve Avc\i $^{\diamondsuit ,\bigstar }$}
\email{merveavci@ymail.com}
\thanks{$^{\bigstar }$Corresponding Author}
\author{Havva Kavurmac\i $^{\diamondsuit }$}
\email{havva.kvrmc@yahoo.com}

\begin{abstract}
In this paper, we establish some integral inequalities for functions whose
second derivatives in absolute value are $\left( \alpha ,m\right) -$ convex.
\end{abstract}

\maketitle

\section{INTRODUCTION}

The following inequality is well known in the literature as the
Hermite-Hadamard integral inequality:%
\begin{equation*}
f\left( \frac{a+b}{2}\right) \leq \frac{1}{b-a}\int_{a}^{b}f(x)dx\leq \frac{%
f(a)+f(b)}{2}
\end{equation*}
where $f:I\subseteq 
\mathbb{R}
\rightarrow 
\mathbb{R}
$ is a convex function on the interval $I$ of real numbers and $a,b\in I$
with $a<b.$

In \cite{M}, V.G. Mihe\c{s}an introduced the class of $(\alpha ,m)-$convex
functions as the following:

The function $f:[0,b]\rightarrow 
\mathbb{R}
$ is said to be $(\alpha ,m)-$convex, where $(\alpha ,m)\in \lbrack
0,1]^{2}, $ if for every $x,y\in \lbrack 0,b]$ and $t\in \lbrack 0,1]$ we
have%
\begin{equation*}
f(tx+m(1-t)y)\leq t^{\alpha }f(x)+m(1-t^{\alpha })f(y).
\end{equation*}

Note that for $(\alpha ,m)\in \left\{ \left( 0,0\right) ,\left( \alpha
,0\right) ,\left( 1,0\right) ,\left( 1,m\right) ,\left( 1,1\right) ,\left(
\alpha ,1\right) \right\} $ one obtains the following classes of functions:
increasing, $\alpha -$starshaped, starshaped, $m-$convex, convex and $\alpha
-$convex.

Denote by $K_{m}^{\alpha }(b)$ the set of all $(\alpha ,m)-$convex functions
on $[0,b]$ which $f(0)\leq 0.$ For recent results and generalizations
concerning $m-$convex and $\left( \alpha ,m\right) -$ convex functions see 
\cite{BOP}, \cite{BPR} and \cite{SSOR}.

In \cite{OAS}, M. Emin \"{O}zdemir, Merve Avci and Erhan Set used the
following lemma in order to establish some inequalities for $m-$ convex
functions.

\begin{lemma}
\label{lemma 1.1} Let $f:I\subset 
\mathbb{R}
\rightarrow 
\mathbb{R}
$ be a twice differentiable mapping on $I^{\circ },a,b\in I$ with $a<b$ and $%
f^{\prime \prime }\in L[a,b].$ Then the following equality holds:%
\begin{equation*}
\frac{f(a)+f(b)}{2}-\frac{1}{b-a}\int_{a}^{b}f(x)dx=\frac{\left( b-a\right)
^{2}}{2}\int_{0}^{1}t\left( 1-t\right) f^{\prime \prime }(ta+(1-t)b)dt.
\end{equation*}
\end{lemma}

In the same paper \cite{OAS}, \"{O}zdemir et al. discussed the following new
results connecting with $m-$convex functions:

\begin{theorem}
\label{teo 1.1} Let $f:I^{\circ }\rightarrow 
\mathbb{R}
,$ where $I^{\circ }\subset \lbrack 0,\infty )$ be a twice differentiable
function on $I^{\circ },a,b\in I$ with $a<b$ and suppose that $f^{\prime
\prime }\in L[a,b].$ If $\left\vert f^{\prime \prime }\right\vert ^{q}$ is $%
m-$ convex on $[a,b]$ for some fixed $q>1$ and $m\in (0,1]$ then the
following inequality holds:%
\begin{equation*}
\frac{f(a)+f(b)}{2}-\frac{1}{b-a}\int_{a}^{b}f(x)dx\leq \frac{\left(
b-a\right) ^{2}}{8}\left( \frac{\Gamma \left( 1+p\right) }{\Gamma \left( 
\frac{3}{2}+p\right) }\right) ^{\frac{1}{p}}\left( \frac{\left\vert
f^{\prime \prime }(a)\right\vert ^{q}+m\left\vert f^{\prime \prime }(\frac{b%
}{m})\right\vert ^{q}}{2}\right) ^{\frac{1}{q}}
\end{equation*}%
where $p=\frac{q}{q-1}.$
\end{theorem}

\begin{corollary}
\label{co 1.1} With the above assumptions given that $\left\vert f^{\prime
\prime }(x)\right\vert \leq K$ on $[a,b],$ and $0<m\leq 1,$ we have the
inequality%
\begin{equation*}
\frac{f(a)+f(b)}{2}-\frac{1}{b-a}\int_{a}^{b}f(x)dx\leq K\frac{\left(
b-a\right) ^{2}}{8}\left( \frac{1+m}{2}\right) ^{\frac{1}{q}}\left( \frac{%
\Gamma \left( 1+p\right) }{\Gamma \left( \frac{3}{2}+p\right) }\right) ^{%
\frac{1}{p}}.
\end{equation*}
\end{corollary}

In \cite{SA}, Sar\i kaya and Aktan obtained the following result concerning
Hermite-Hadamard's inequality for functions whose second derivative in
absolute value is convex as follows:

\begin{theorem}
\label{teo 1.2}Let $I\subseteq 
\mathbb{R}
$ be an open interval, $a,b\in I$ with $a<b$ and $f:I\rightarrow 
\mathbb{R}
$ be twice differentiable mapping such that $f^{\prime \prime }$ is
integrable and $0\leq \lambda \leq 1.$ If $\left\vert f^{\prime \prime
}\right\vert $ is a convex function on $[a,b],$ then the following
inequalities hold:%
\begin{eqnarray*}
&&\left\vert \left( \lambda -1\right) f\left( \frac{a+b}{2}\right) -\lambda 
\frac{f(a)+f(b)}{2}+\frac{1}{b-a}\int_{a}^{b}f(x)dx\right\vert \\
&\leq &\left\{ 
\begin{array}{l}
\frac{\left( b-a\right) ^{2}}{12}[\left( \lambda ^{4}+\left( 1+\lambda
\right) \left( 1-\lambda \right) ^{3}+\frac{5\lambda -3}{4}\right)
\left\vert f^{\prime \prime }(a)\right\vert \text{ \ \ \ for }0\leq \lambda
\leq \frac{1}{2} \\ 
\\ 
+\left( \lambda ^{4}+\left( 2-\lambda \right) \lambda ^{3}+\frac{1-3\lambda 
}{4}\right) \left\vert f^{\prime \prime }(b)\right\vert ], \\ 
\\ 
\frac{\left( b-a\right) ^{2}\left( 3\lambda -1\right) }{48}\left[ \left\vert
f^{\prime \prime }(a)\right\vert +\left\vert f^{\prime \prime
}(b)\right\vert \text{ }\right] ,\text{ \ \ \ \ \ \ \ \ \ \ \ \ \ \ \ \ \ \
\ \ \ for }\frac{1}{2}\leq \lambda \leq 1.%
\end{array}%
\right.
\end{eqnarray*}
\end{theorem}

In Theorem \ref{teo 1.2}, if we choose $\lambda =1$ we have%
\begin{equation}
\left\vert \frac{1}{b-a}\int_{a}^{b}f(x)dx-\frac{f(a)+f(b)}{2}\right\vert
\leq \frac{\left( b-a\right) ^{2}}{12}\left[ \frac{\left\vert f^{\prime
\prime }(a)\right\vert +\left\vert f^{\prime \prime }(b)\right\vert }{2}%
\right] .  \label{0}
\end{equation}

The aim of this paper is to establish some inequalities like those given in 
\cite{OAS}, but now for $(\alpha ,m)-$convex functions. That is, this study
is a continuation of \cite{OAS}. In order to obtain our results, we modified
Lemma \ref{lemma 1.1} given in the \cite{OAS}$.$

\section{INEQUALITIES FOR $\left( \protect\alpha ,m\right) -$ CONVEX
FUNCTIONS}

\begin{lemma}
\label{lemma 2.1} Let $f:I\subseteq 
\mathbb{R}
\rightarrow 
\mathbb{R}
$ be a twice differentiable mapping on $I^{\circ }$ where $a,b\in I$ with $%
a<b$ and $m\in (0,1].$ If $f^{\prime \prime }\in L[a,b],$ then the following
equality holds:%
\begin{equation*}
\frac{f(a)+f(mb)}{2}-\frac{1}{mb-a}\int_{a}^{mb}f(x)dx=\frac{\left(
mb-a\right) ^{2}}{2}\int_{0}^{1}\left( t-t^{2}\right) f^{\prime \prime
}(ta+m(1-t)b)dt.
\end{equation*}
\end{lemma}

A simple proof of the equality can be done by performing an integration by
parts in the integrals from the right side and changing the variable. The
details are left to the interested reader.

\begin{theorem}
\label{teo 2.1} Let $f:I\subset \lbrack 0,b^{\ast }]\rightarrow 
\mathbb{R}
$ be a differentiable mapping on $I^{\circ \text{ }}$ such that $f^{\prime
\prime }\in L[a,b]$ where $a,b$ $\in I$ with $a<b,$ $b^{\ast }>0.$ If $%
\left\vert f^{\prime \prime }\right\vert ^{q}$ is $(\alpha ,m)-$convex on $%
[a,b]$ for $(\alpha ,m)\in \lbrack 0,1]^{2},$ $q\geq 1,$ then the following
inequality holds:%
\begin{eqnarray*}
&&\left\vert \frac{f(a)+f(mb)}{2}-\frac{1}{mb-a}\int_{a}^{mb}f(x)dx\right%
\vert \\
&\leq &\frac{\left( mb-a\right) ^{2}}{2}\left( \frac{1}{6}\right) ^{1-\frac{1%
}{q}} \\
&&\times \left[ \left\vert f^{\prime \prime }(a)\right\vert ^{q}\frac{1}{%
\left( \alpha +2\right) \left( \alpha +3\right) }+m\left\vert f^{\prime
\prime }(b)\right\vert ^{q}\left( \frac{1}{6}-\frac{1}{\left( \alpha
+2\right) \left( \alpha +3\right) }\right) \right] ^{\frac{1}{q}}.
\end{eqnarray*}
\end{theorem}

\begin{proof}
Suppose that $q=1.$ From Lemma \ref{lemma 2.1} and using the $(\alpha ,m)-$%
convexity of $\left\vert f^{\prime \prime }\right\vert $ we have%
\begin{eqnarray}
&&\left\vert \frac{f(a)+f(mb)}{2}-\frac{1}{mb-a}\int_{a}^{mb}f(x)dx\right%
\vert  \label{1} \\
&\leq &\frac{\left( mb-a\right) ^{2}}{2}\int_{0}^{1}\left( t-t^{2}\right)
\left\vert f^{\prime \prime }(ta+m(1-t)b)\right\vert dt  \notag \\
&\leq &\frac{\left( mb-a\right) ^{2}}{2}\int_{0}^{1}\left( t-t^{2}\right) %
\left[ t^{\alpha }\left\vert f(a)\right\vert +m(1-t^{\alpha })\left\vert
f(b)\right\vert \right] dt  \notag \\
&=&\frac{\left( mb-a\right) ^{2}}{2}\left[ \left\vert f^{\prime \prime
}(a)\right\vert \frac{1}{\left( \alpha +2\right) \left( \alpha +3\right) }%
+m\left\vert f^{\prime \prime }(a)\right\vert \left( \frac{1}{6}-\frac{1}{%
\left( \alpha +2\right) \left( \alpha +3\right) }\right) \right]  \notag
\end{eqnarray}%
which completes the proof for $q=1.$

Suppose now that $q>1.$ From Lemma \ref{lemma 2.1} and using the H\"{o}%
lder's integral inequality for $q>1,$ we have%
\begin{eqnarray}
&&\int_{0}^{1}\left( t-t^{2}\right) \left\vert f^{\prime \prime
}(ta+m(1-t)b)\right\vert dt  \notag \\
&=&\int_{0}^{1}\left( t-t^{2}\right) ^{1-\frac{1}{q}}\left( t-t^{2}\right) ^{%
\frac{1}{q}}\left\vert f^{\prime \prime }(ta+m(1-t)b)\right\vert dt  \notag
\\
&\leq &\left[ \int_{0}^{1}\left( t-t^{2}\right) dt\right] ^{1-\frac{1}{q}}%
\left[ \int_{0}^{1}\left( t-t^{2}\right) \left\vert f^{\prime \prime
}(ta+m(1-t)b)\right\vert ^{q}dt\right] ^{\frac{1}{q}}  \label{2}
\end{eqnarray}%
where $\frac{1}{p}+\frac{1}{q}=1.$

Since $\left\vert f^{\prime \prime }\right\vert ^{q}$ is $(\alpha ,m)-$
convex on $[a,b]$ we know that for every $t\in \lbrack 0,1]$%
\begin{equation}
\left\vert f^{\prime \prime }(ta+m(1-t)b)\right\vert ^{q}\leq t^{\alpha
}\left\vert f^{\prime \prime }(a)\right\vert ^{q}+m\left( 1-t^{\alpha
}\right) \left\vert f^{\prime \prime }(b)\right\vert ^{q}.  \label{3}
\end{equation}%
From (\ref{1})-(\ref{3}) we have%
\begin{eqnarray*}
&&\left\vert \frac{f(a)+f(mb)}{2}-\frac{1}{mb-a}\int_{a}^{mb}f(x)dx\right%
\vert \\
&\leq &\frac{\left( mb-a\right) ^{2}}{2}\left[ \int_{0}^{1}\left(
t-t^{2}\right) dt\right] ^{1-\frac{1}{q}}\left[ \int_{0}^{1}\left(
t-t^{2}\right) \left\vert f^{\prime \prime }(ta+m(1-t)b)\right\vert ^{q}dt%
\right] ^{\frac{1}{q}} \\
&\leq &\frac{\left( mb-a\right) ^{2}}{2}\left[ \int_{0}^{1}\left(
t-t^{2}\right) dt\right] ^{1-\frac{1}{q}} \\
&&\times \left[ \int_{0}^{1}\left( t-t^{2}\right) \left[ t^{\alpha
}\left\vert f^{\prime \prime }(a)\right\vert ^{q}+m\left( 1-t^{\alpha
}\right) \left\vert f^{\prime \prime }(b)\right\vert ^{q}\right] dt\right] ^{%
\frac{1}{q}} \\
&=&\frac{\left( mb-a\right) ^{2}}{2}\left( \frac{1}{6}\right) ^{1-\frac{1}{q}%
} \\
&&\times \left[ \left\vert f^{\prime \prime }(a)\right\vert ^{q}\frac{1}{%
\left( \alpha +2\right) \left( \alpha +3\right) }+m\left\vert f^{\prime
\prime }(b)\right\vert ^{q}\left( \frac{1}{6}-\frac{1}{\left( \alpha
+2\right) \left( \alpha +3\right) }\right) \right] ^{\frac{1}{q}}
\end{eqnarray*}%
which is the required.
\end{proof}

\begin{remark}
\label{rem 2.1} If in Theorem \ref{teo 2.1} we choose $m=\alpha =q=1,$ we
obtain 
\begin{equation*}
\left\vert \frac{f(a)+f(b)}{2}-\frac{1}{b-a}\int_{a}^{b}f(x)dx\right\vert
\leq \frac{\left( b-a\right) ^{2}}{24}\left[ \left\vert f^{\prime \prime
}(a)\right\vert +\left\vert f^{\prime \prime }(b)\right\vert \right]
\end{equation*}%
which is the inequality in (\ref{0}).
\end{remark}

\begin{theorem}
\label{teo 2.2} With the assumptions of Theorem \ref{teo 2.1} the following
inequality holds:%
\begin{eqnarray*}
&&\left\vert \frac{f(a)+f(mb)}{2}-\frac{1}{mb-a}\int_{a}^{mb}f(x)dx\right%
\vert \\
&\leq &\frac{\left( mb-a\right) ^{2}}{8}\left( \frac{\Gamma \left(
1+p\right) }{\Gamma \left( \frac{3}{2}+p\right) }\right) ^{\frac{1}{p}%
}\left( \left\vert f^{\prime \prime }(a)\right\vert ^{q}\frac{1}{\alpha +1}%
+m\left\vert f^{\prime \prime }(b)\right\vert ^{q}\left( \frac{\alpha }{%
\alpha +1}\right) \right) ^{\frac{1}{q}}.
\end{eqnarray*}
\end{theorem}

\begin{proof}
From Lemma \ref{lemma 2.1} and using the H\"{o}lder's integral inequality,
we obtain%
\begin{eqnarray*}
&&\left\vert \frac{f(a)+f(mb)}{2}-\frac{1}{mb-a}\int_{a}^{mb}f(x)dx\right%
\vert \\
&\leq &\frac{\left( mb-a\right) ^{2}}{2}\int_{0}^{1}t(1-t)\left\vert
f^{\prime \prime }(ta+m(1-t)b)\right\vert dt \\
&\leq &\frac{\left( mb-a\right) ^{2}}{2}\left( \int_{0}^{1}\left(
t-t^{2}\right) ^{p}dt\right) ^{\frac{1}{p}}\left( \int_{0}^{1}\left\vert
f^{\prime \prime }(ta+m(1-t)b)\right\vert ^{q}dt\right) ^{\frac{1}{q}} \\
&\leq &\frac{\left( mb-a\right) ^{2}}{2}\left( \int_{0}^{1}\left(
t-t^{2}\right) ^{p}dt\right) ^{\frac{1}{p}}\left( \int_{0}^{1}\left[
t^{\alpha }\left\vert f^{\prime \prime }(a)\right\vert ^{q}+m\left(
1-t^{\alpha }\right) \left\vert f^{\prime \prime }(b)\right\vert ^{q}\right]
dt\right) ^{\frac{1}{q}} \\
&=&\frac{\left( mb-a\right) ^{2}}{2}\left( \frac{2^{-1-2p}\sqrt{\pi }\Gamma
\left( 1+p\right) }{\Gamma \left( \frac{3}{2}+p\right) }\right) ^{\frac{1}{p}%
}\left( \left\vert f^{\prime \prime }(a)\right\vert ^{q}\frac{1}{\alpha +1}%
+m\left\vert f^{\prime \prime }(b)\right\vert ^{q}\left( \frac{\alpha }{%
\alpha +1}\right) \right) ^{\frac{1}{q}} \\
&=&\frac{\left( mb-a\right) ^{2}}{2}\frac{\left( \sqrt{\pi }\right) ^{\frac{1%
}{p}}}{2^{\frac{1}{p}}2^{2}}\left( \frac{\Gamma \left( 1+p\right) }{\Gamma
\left( \frac{3}{2}+p\right) }\right) ^{\frac{1}{p}}\left( \left\vert
f^{\prime \prime }(a)\right\vert ^{q}\frac{1}{\alpha +1}+m\left\vert
f^{\prime \prime }(b)\right\vert ^{q}\left( \frac{\alpha }{\alpha +1}\right)
\right) ^{\frac{1}{q}}
\end{eqnarray*}%
and since $\sqrt{\pi }<2,$ we have%
\begin{eqnarray*}
&&\left\vert \frac{f(a)+f(mb)}{2}-\frac{1}{mb-a}\int_{a}^{mb}f(x)dx\right%
\vert \\
&\leq &\frac{\left( mb-a\right) ^{2}}{8}\left( \frac{\Gamma \left(
1+p\right) }{\Gamma \left( \frac{3}{2}+p\right) }\right) ^{\frac{1}{p}%
}\left( \left\vert f^{\prime \prime }(a)\right\vert ^{q}\frac{1}{\alpha +1}%
+m\left\vert f^{\prime \prime }(b)\right\vert ^{q}\left( \frac{\alpha }{%
\alpha +1}\right) \right) ^{\frac{1}{q}},
\end{eqnarray*}%
where $\frac{1}{p}+\frac{1}{q}=1.$

We note that, the Beta and Gamma function (see \cite{GR}),%
\begin{equation*}
\beta \left( x,y\right) =\int_{0}^{1}t^{x-1}\left( 1-t\right) ^{y-1}dt,\text{
\ \ \ }x,y>0,\text{ \ \ \ \ \ }\Gamma (x)=\int_{0}^{\infty }e^{-t}t^{x-1}dt,%
\text{ \ \ \ }x>0,
\end{equation*}%
are used to evaluate the integral%
\begin{equation*}
\int_{0}^{1}\left( t-t^{2}\right) ^{p}dt=\int_{0}^{1}t^{p}\left( 1-t\right)
^{p}dt=\beta \left( p+1,p+1\right)
\end{equation*}%
where%
\begin{equation*}
\beta (x,x)=2^{1-2x}\beta (\frac{1}{2},x)\text{ \ \ \ and \ \ \ }\beta (x,y)=%
\frac{\Gamma (x)\Gamma (y)}{\Gamma (x+y)},\text{ thus we can obtain that}
\end{equation*}%
\begin{equation*}
\beta (p+1,p+1)=2^{1-2(p+1)}\beta (\frac{1}{2},p+1)=2^{1-2(p+1)}\frac{\Gamma
(\frac{1}{2})\Gamma (p+1)}{\Gamma (\frac{3}{2}+p)},
\end{equation*}%
and $\Gamma \left( \frac{1}{2}\right) =\sqrt{\pi },$ which completes the
proof.
\end{proof}

\begin{remark}
\label{rem 2.3} Suppose that all the assumptions of Theorem \ref{teo 2.2}
are satisfied with $\left\vert f^{\prime \prime }\right\vert \leq K.$ If we
choose $m=\alpha =1,$ we have the inequality in Corollary \ref{co 1.1}.
\end{remark}

\begin{theorem}
\label{teo 2.3} Let $f:I\subset \lbrack 0,b^{\ast }]\rightarrow 
\mathbb{R}
$ be a differentiable mapping on $I^{\circ \text{ }}$ such that $f^{\prime
\prime }\in L[a,b]$ where $a,b$ $\in I$ with $a<b,$ $b^{\ast }>0.$ If $%
\left\vert f^{\prime \prime }\right\vert ^{q}$ is $(\alpha ,m)-$ convex on $%
[a,b]$ for $(\alpha ,m)\in \lbrack 0,1]^{2},$ $q>1,$ $\frac{1}{p}+\frac{1}{q}%
=1,$ then the following inequality holds:%
\begin{eqnarray*}
&&\left\vert \frac{f(a)+f(mb)}{2}-\frac{1}{mb-a}\int_{a}^{mb}f(x)dx\right%
\vert \\
&\leq &\frac{\left( mb-a\right) ^{2}}{2} \\
&&\times \left( \left\vert f^{\prime \prime }(a)\right\vert ^{q}\beta \left(
\alpha +1,q+1\right) +m\left\vert f^{\prime \prime }(b)\right\vert
^{q}\left( \frac{1}{q+1}-\beta \left( \alpha +1,q+1\right) \right) \right) ^{%
\frac{1}{q}}.
\end{eqnarray*}
\end{theorem}

\begin{proof}
From Lemma \ref{lemma 2.1} and using the well known H\"{o}lder's integral
inequality we have%
\begin{eqnarray*}
&&\left\vert \frac{f(a)+f(mb)}{2}-\frac{1}{mb-a}\int_{a}^{mb}f(x)dx\right%
\vert \\
&\leq &\frac{\left( mb-a\right) ^{2}}{2}\int_{0}^{1}t(1-t)\left\vert
f^{\prime \prime }(ta+m(1-t)b)\right\vert dt \\
&\leq &\frac{\left( mb-a\right) ^{2}}{2}\left( \int_{0}^{1}t^{p}dt\right) ^{%
\frac{1}{p}}\left( \int_{0}^{1}\left( 1-t\right) ^{q}\left\vert f^{\prime
\prime }(ta+m(1-t)b)\right\vert ^{q}dt\right) ^{\frac{1}{q}} \\
&\leq &\frac{\left( mb-a\right) ^{2}}{2}\left( \int_{0}^{1}t^{p}dt\right) ^{%
\frac{1}{p}} \\
&&\times \left( \int_{0}^{1}\left( 1-t\right) ^{q}\left[ t^{\alpha
}\left\vert f^{\prime \prime }(a)\right\vert ^{q}+m\left( 1-t^{\alpha
}\right) \left\vert f^{\prime \prime }(b)\right\vert ^{q}\right] dt\right) ^{%
\frac{1}{q}} \\
&=&\frac{\left( mb-a\right) ^{2}}{2}\left( \frac{1}{p+1}\right) ^{\frac{1}{p}%
} \\
&&\times \left( \left\vert f^{\prime \prime }(a)\right\vert ^{q}\beta \left(
\alpha +1,q+1\right) +m\left\vert f^{\prime \prime }(b)\right\vert
^{q}\left( \frac{1}{q+1}-\beta \left( \alpha +1,q+1\right) \right) \right) ^{%
\frac{1}{q}}.
\end{eqnarray*}%
Since $\frac{1}{2}<\left( \frac{1}{p+1}\right) ^{\frac{1}{p}}<1,$ we obtain%
\begin{eqnarray*}
&&\left\vert \frac{f(a)+f(mb)}{2}-\frac{1}{mb-a}\int_{a}^{mb}f(x)dx\right%
\vert \\
&\leq &\frac{\left( mb-a\right) ^{2}}{2}\left( \left\vert f^{\prime \prime
}(a)\right\vert ^{q}\beta \left( \alpha +1,q+1\right) +m\left\vert f^{\prime
\prime }(b)\right\vert ^{q}\left( \frac{1}{q+1}-\beta \left( \alpha
+1,q+1\right) \right) \right) ^{\frac{1}{q}}.
\end{eqnarray*}
\end{proof}

\begin{theorem}
\label{teo 2.4} Let $f:I\subset \lbrack 0,b^{\ast }]\rightarrow 
\mathbb{R}
$ be a differentiable mapping on $I^{\circ \text{ }}$ such that $f^{\prime
\prime }\in L[a,b]$ where $a,b$ $\in I$ with $a<b,$ $b^{\ast }>0.$ If $%
\left\vert f^{\prime \prime }\right\vert ^{q}$ is $(\alpha ,m)-$ convex on $%
[a,b]$ for $(\alpha ,m)\in \lbrack 0,1]^{2},$ and for some fixed $q\geq 1,$
then the following inequality holds:%
\begin{eqnarray*}
&&\left\vert \frac{f(a)+f(mb)}{2}-\frac{1}{mb-a}\int_{a}^{mb}f(x)dx\right%
\vert \\
&\leq &\frac{\left( mb-a\right) ^{2}}{2}\left( \frac{1}{2}\right) ^{1-\frac{1%
}{q}} \\
&&\times \left( \left( \left\vert f^{\prime \prime }(a)\right\vert ^{q}\beta
\left( \alpha +2,q+1\right) +m\left\vert f^{\prime \prime }(b)\right\vert
^{q}\left( \frac{1}{\left( q+1\right) \left( q+2\right) }-\beta \left(
\alpha +2,q+1\right) \right) \right) \right) ^{\frac{1}{q}}.
\end{eqnarray*}
\end{theorem}

\begin{proof}
From Lemma \ref{lemma 2.1} and using the well-known power-mean integral
inequality we have%
\begin{eqnarray*}
&&\left\vert \frac{f(a)+f(mb)}{2}-\frac{1}{mb-a}\int_{a}^{mb}f(x)dx\right%
\vert \\
&\leq &\frac{\left( mb-a\right) ^{2}}{2}\int_{0}^{1}t(1-t)\left\vert
f^{\prime \prime }(ta+m(1-t)b)\right\vert dt \\
&\leq &\frac{\left( mb-a\right) ^{2}}{2}\left( \int_{0}^{1}tdt\right) ^{1-%
\frac{1}{q}}\left( \int_{0}^{1}t(1-t)^{q}\left\vert f^{\prime \prime
}(ta+m(1-t)b)\right\vert ^{q}dt\right) ^{\frac{1}{q}} \\
&\leq &\frac{\left( mb-a\right) ^{2}}{2}\left( \int_{0}^{1}tdt\right) ^{1-%
\frac{1}{q}}\left( \int_{0}^{1}t\left( 1-t\right) ^{q}\left[ t^{\alpha
}\left\vert f^{\prime \prime }(a)\right\vert ^{q}+m\left( 1-t^{\alpha
}\right) \left\vert f^{\prime \prime }(b)\right\vert ^{q}\right] dt\right) ^{%
\frac{1}{q}} \\
&=&\frac{\left( mb-a\right) ^{2}}{2}\left( \frac{1}{2}\right) ^{1-\frac{1}{q}%
} \\
&&\times \left( \left( \left\vert f^{\prime \prime }(a)\right\vert ^{q}\beta
\left( \alpha +2,q+1\right) +m\left\vert f^{\prime \prime }(b)\right\vert
^{q}\left( \frac{1}{\left( q+1\right) \left( q+2\right) }-\beta \left(
\alpha +2,q+1\right) \right) \right) \right) ^{\frac{1}{q}}
\end{eqnarray*}%
which completes the proof.
\end{proof}

\end{document}